\documentclass[a4paper,10pt,reqno]{amsart}
\usepackage{}
\usepackage{bbding}

\usepackage{amsfonts}
\usepackage{mathrsfs}
\usepackage{amsmath}
\usepackage{amsmath,graphics}%diagrams,
\usepackage{amssymb}
\usepackage{indentfirst,latexsym,bm,amsmath,pstricks,amssymb,amsthm,graphicx,fancyhdr,float,color}
\usepackage[all]{xy}
\usepackage{amsthm}
\usepackage{amscd}
\usepackage{hyperref}
\usepackage[all]{hypcap}

\newtheorem{theorem}{Theorem}[section]

\newtheorem{lemma}[theorem]{Lemma}
\newtheorem{conjecture}[theorem]{Conjecture}

\newtheorem{corollary}[theorem]{Corollary}
\newtheorem{proposition}[theorem]{Proposition}

\theoremstyle{definition}
\newtheorem{definition}[theorem]{Definition}

\theoremstyle{remark}
\newtheorem{remark}[theorem]{Remark}
\newtheorem{remarks}[theorem]{Remarks}

\numberwithin{equation}{section}

\def\wt{\widetilde}

\def\({$($}
\def\){$)$}
\def\chit{\chi_{\rm top}}
\def\bbp{\mathbb P}

\def\Pic{\text{{\rm Pic\,}}}

\begin{document}

\title[Minimal number of singular fibers with non-compact Jacobians]
{On the minimal number of singular fibers with non-compact Jacobians for families of curves over $\bbp^1$}

%    Information for first author
\author{Xin Lu}
%    Address of record for the research reported here
\address{Department of Mathematics, Shanghai Key Laboratory of PMMP, East China Normal University, Dongchuan RD 500, Shanghai 200241, P. R. of China}
\curraddr{Institut f\"ur Mathematik, Universit\"at Mainz, Mainz, 55099, Germany}

\email{lvxinwillv@gmail.com}
%    \thanks will become a 1st page footnote.
\thanks{This work is supported by SFB/Transregio 45 Periods, Moduli Spaces and Arithmetic of Algebraic Varieties of the DFG (Deutsche Forschungsgemeinschaft), and by
 National Key Basic Research Program of China (Grant No. 2013CB834202) and NSF of China (Grant No. 11231003).}

%    Information for second named author
\author{Sheng-Li Tan}
\address{Department of Mathematics, East China Normal University, Dongchuan RD 500, Shanghai 200241, P. R. of China}
\email{sltan@math.ecnu.edu.cn}

%    Information for third named author
\author{Wan-Yuan Xu}
\address{Department of Mathematics, Shanghai Key Laboratory of PMMP, East China Normal University, Dongchuan RD 500, Shanghai 200241, P. R. of China}
\curraddr{Emmy Noether Research Institute for Mathematics, Bar-Ilan University, Ramat Gan, 52900, Israel}
\email{xwydy1988@126.com}

%    Information for fourth named author
\author{Kang Zuo}
\address{Institut f\"ur Mathematik, Universit\"at Mainz, Mainz, 55099, Germany}
\email{zuok@uni-mainz.de}

%    General info
%14J25 Special surfaces
%14D06 Fibrations, degenerations
%14H10 Families, moduli (algebraic)///curves

%\date{September 19, 2014}

%\dedicatory{This paper is dedicated to our advisors.}


\maketitle

\begin{abstract}
\addcontentsline{toc}{section}{Abstract}
Let $f:X \to \mathbb{P}^1$ be a non-isotrivial family of semi-stable curves of genus $g\geq 1$ defined over an algebraically closed field $k$
with $s_{nc}$ singular fibers whose Jacobians are non-compact.
We prove that $s_{nc}\geq 5$ if $k=\mathbb C$ and $g\geq 5$; we also prove that $s_{nc}\geq 4$ if ${\rm char}~k>0$ and the relative Jacobian of $f$ is non-smooth.
\end{abstract}

\section{Introduction}
Let $f:X \to \mathbb{P}^1$ be a non-isotrivial family of semi-stable curves of genus $g\geq 1$
defined over an algebraically closed field $k$.
Denote by $s$ the number of the singular fibers
and $s_{nc}$ the number of the singular fibers whose Jacobians are non-compact.
It is a classical problem to determine lower bounds for $s$ and $s_{nc}$ (cf. \cite{Be81, Sz81}).

When $k=\mathbb C$, it was first proved by Beauville (cf. \cite{Be81}) that $s\geq 4$.
He also conjectured that $s\geq 5$ if $g\geq 2$, which was confirmed by the second named author \cite{Tan95}
based on the strict canonical class inequality for a family of curves.
There exist examples with $s=5$ and $g=2$ or $3$.
It is conjectured that $s\geq 6$ if $g$ is sufficiently large.
This was confirmed by Tu, Zamora and the second named author in \cite{TTZ05}
under an additional assumption that the Kodaira dimension of $X$ is non-negative.

The lower bound for $s_{nc}$ is much more mysterious.
For a long time, it was only known that $s_{nc}\geq 4$ holds for $g=1$ by Beauville's result (cf. \cite{Be81}),
due to the simple fact that $s_{nc}=s$ in the case.
In \cite{V-Z04},  Viehweg and the last named author showed that $s_{nc}\geq4$ for all $g\geq 1$
by using deep theories such as Simpson's correspondence.
They also proved that if $s_{nc}=4$, then the relative Jacobian of $f$
is the universal family of abelian varieties over a Shimura curve,
which implies in particular that such a family is actually defined over some number field.
Very recently, there is another proof of $s_{nc}\geq4$ given in \cite{LTYZ} based on an inequality
for the Hodge number $h^{1,1}(X)$, which seems more natural.
Examples of semi-stable families over $\mathbb P^1$ with $s_{nc}=4$
exist only for $g=1$, $2$ and $3$ (cf. \cite{Be2,LTYZ,luzuo14}).
No such examples have been found for high genera.
In fact, it was conjectured by Viehweg and the last named author (cf. \cite{V-Z04}) that
\begin{conjecture}[$k=\mathbb C$]\label{conjectures_nc>=5}
$s_{nc}\geq5$ for $g\gg0$.
\end{conjecture}
When $X$ is regular (i.e., $q:=h^1(X,\mathcal {O}_X)=0$),
the above conjecture was confirmed by Kukulies \cite{Ku10},
based on Viehweg-Zuo's characterization of such families and Sato-Tate's conjecture for modular curves;
see also \cite{luzuo14} for an alternate proof using surface theory.

The first main result of our paper is to prove the above conjecture.
\begin{theorem}[$k=\mathbb C$]\label{mainthm1}
Conjecture {\rm\ref{conjectures_nc>=5}} is true.
More precisely, $s_{nc}\geq5$ if $g\geq5$.
\end{theorem}
\begin{remarks}
(i). Since $s\geq s_{nc}$, our result implies Beauville's conjecture for $g\geq 5$,
which was proved earlier by the second named author \cite{Tan95}.

(ii). If one removes the assumption on the semi-stability of $f$,
then it is only known that $s\geq s_{nc}\geq 3$ (cf. \cite{Be81}).
Even if we require two of the singular fibers to be semi-stable,
there still exist infinitely many examples with exactly three singular fibers (cf. \cite{glt13}).
\end{remarks}

Our next purpose is to consider the similar question in the positive characteristic case.
When $k$ is of characteristic $p>0$,
Szpiro \cite{Sz81} first proved that $s\geq 3$ and he also showed that $s\geq 4$ under the assumption that $X$ is of general type.
Later, Nguyen proved in \cite{Ngu98} that $s\geq4$ without any assumption on the total surface $X$,
which is an analogue of the Beauville's result (cf. \cite{Be81}) in the positive characteristic case.
Instead of finding the lower bound on $s$, we prove a lower bound on $s_{nc}$ if the relative Jacobian of $f$ is non-smooth over $\mathbb P^1$.
More precisely,
\begin{theorem}[char $k>0$]\label{thm2}
Assume that the relative Jacobian of $f$ is non-smooth. Then $s_{nc}\geq4$.
\end{theorem}

\begin{remarks}
(i). Similarly, the above theorem implies Nguyen's result $s\geq 4$ (cf. \cite{Ngu98}) if the relative Jacobian of $f$ is non-smooth due to the fact that $s\geq s_{nc}$.

(ii). The assumption that the relative Jacobian of $f$ is non-smooth is crucial in our theorem.
Actually, Moret-Bailly \cite{M-B} constructed a semi-stable family of genus 2 curves over $\mathbb{P}^1$ defined over $k$ with $s_{nc}=0$,
i.e., the relative Jacobian of $f$ is a smooth family of abelian surfaces.

(iii). If $s_{nc}$ is indeed equal to $4$, $X$ would be very special. Actually, the numerical data of the surface $X$ is obtained in Section \ref{sectionthm2}.
\end{remarks}

Our paper is organized as follows. In Section \ref{sectionpre}, we review some preliminaries about families of curves.
The main results, Theorems \ref{mainthm1} and \ref{thm2}, are proved in Sections \ref{sectionthm1} and \ref{sectionthm2} respectively.

\section{Preliminaries}\label{sectionpre}
In this section, we would like to review some basic facts on families of semi-stable curves
and fix the notations which will be used freely. For more details we refer to \cite{Ba01,bhpv,Se58a}.

Let $f:X \to B$ be a relatively minimal fibration of genus $g\geq 1$ defined over an algebraic closed field $k$,
i.e., there is no $(-1)$-curve contained in fibers of $f$.
A fiber $F$ is called semi-stable (resp. stable) if it is a reduced nodal curve,
and every smooth rational component intersects the rest part of $F$ at least two (resp. three) points.
$f$ is said to be semi-stable if all singular fibers of $f$ are semi-stable,
and isotrivial if $f$ becomes trivial after a base change.
Unless stated otherwise, the total surface $X$ is supposed to be smooth.

Let $K_X$ be the canonical divisor of $X$,
$\omega_{X/B}=\omega_X\otimes f^*\omega^{\vee}_{B}$ the relative canonical sheaf,
$\chi(\mathcal {O}_X)=\sum(-1)^i h^i(X, \mathcal O_X)$ the Euler characteristic of the structure sheaf,
and $\chit(X)=\sum(-1)^i b_i$ the topological Euler characteristic of $X$, where $b_i$'s are the Betti numbers of $X$.
We should mention that, for an arbitrary field $k$, the Betti numbers are defined by  $l$-adic cohomology, i.e., $b_i(X)=\dim H_{\acute{e}t}^i(X,\mathbb{Q}_l)$ (here $l\neq \text{char~} k$);
and for $k=\mathbb{C}$, the Betti numbers defined by  $l$-adic cohomology coincide with that given by singular cohomology.
Moreover, one has $b_i=b_{4-i}$ by Poincar\'e duality theorem.
The relative invariants of $f$ are defined as follows:
\begin{equation}\label{eqnrelativeinvs1}
\chi_f:=\deg f_*\omega_{X/B},\qquad
K_f^2:=\omega^2_{X/B},\qquad\delta_f:=\sum_{F \text{~is singular}} \delta(F),
\end{equation}
where
$$\delta(F)=\chit(F)-(2-2g).$$
Let $b=g(B)$ the genus of $B$. It is well-known that
\begin{equation}\label{eqnrelativeinvs2}
\begin{cases}
\chi_f&=~\chi(\mathcal {O}_X)-(g-1)(b-1),\\[0.1cm]
K_f^2&=~K_X^2-8(g-1)(b-1),\\[0.1cm]
\delta_f&=~\chit(X)-4(g-1)(b-1).
\end{cases}
\end{equation}
These relative invariants are all non-negative.
%And $\delta_f=0$ iff $f$ is smooth;
%$\chi_f=0$ (equivalently $K_f^2=0$) iff $f$ is smooth and isotrivial.
They also satisfy the Noether's formula
\begin{equation}\label{eqnnoetherformula}
12\chi_f=K_f^2+\delta_f.
\end{equation}

Note that when $f$ is semi-stable,
$\delta(F)$ is nothing but the number of nodes contained in $F$ for any singular fiber of $f$.
From now till the end of the section, we assume that $f$ is semi-stable,
and would like to give a more concrete description of the nodes of a singular fiber $F$.
\begin{definition}\label{deftypeinodeinF}
(i). Assume $F=\sum\limits_{i=1}^{l(F)} \Gamma_i$, and $\wt \Gamma_i$ is the normalization of $\Gamma_i$.
Then the geometrical genus of $F$ is defined to be $g(F)=\sum\limits_{i=1}^{l(F)} g(\wt \Gamma_i)$.

(ii). Let $q\in F$ be a node of $F$.
Then $q$ is said to be  of type $i\in [1,  g/2]$ (resp. 0) if the
partial normalization of $F$ at $q$ consists of two connected
components of arithmetic genera $i$ and $g-i$ (resp. is connected).
Denote by $\delta_i(F)$ the number of nodes of type $i$ in $F$.
\end{definition}

\begin{lemma}
For any semi-stable fiber $F=\sum\limits_{i=1}^{l(F)} \Gamma_i$, one has
\begin{equation}\label{eqndelta(F)=}
\delta(F)=g-g(F)+l(F)-1.
\end{equation}
\end{lemma}
\begin{proof}
Let $\wt \Gamma_i$ be the normalization of $\Gamma_i$, and $p_a(\Gamma_i)$ \big(resp. $n(\Gamma_i)$\big) be the arithmetic genus (resp. the number of nodes) of $\Gamma_i$. Then
\begin{equation*}
\begin{aligned}
&g(F)=\sum_{i=1}^{l(F)}g(\widetilde{\Gamma}_i)=\sum_{i=1}^{l(F)}\big(p_a(\Gamma_i)-n(\Gamma_i)\big),&\quad&
p_a(\Gamma_i)=1+\frac{\Gamma_i^2+K_X\cdot\Gamma_i}{2},\\
&\delta(F)=\sum_{i<j}\Gamma_i\cdot \Gamma_j + \sum_{i}^{l(F)} n(\Gamma_i),&&
0=F^2=\sum_{i}^{l(F)} \Gamma_i^2 + 2\sum_{i<j}\Gamma_i\cdot \Gamma_j.
\end{aligned}
\end{equation*}
Combining the above equalities with the fact that $\sum\limits_{i}^{l(F)} K_X\cdot\Gamma_i=K_X \cdot F =2g-2$, one obtains \eqref{eqndelta(F)=}.
\end{proof}
Note that $F$ has a compact Jacobian if $g(F)=g$.
Hence by the above lemma,
$$\begin{aligned}
&\text{$F$ has a compact Jacobian,} &\Longleftrightarrow&\quad \delta(F)=l(F)-1,\qquad\\
\Longleftrightarrow\quad&\text{the dual graph of $F$ is a tree,} &\Longleftrightarrow&\quad \delta_0(F)=0.
\end{aligned}
$$

Let $\Upsilon \to \Delta$ (resp. $\Upsilon_{ct} \to \Delta_{ct}$, resp. $\Upsilon_{nc} \to \Delta_{nc}$)
be the singular fibers (resp. singular fibers with compact Jacobians, resp. singular fibers with non-compact Jacobians)
of $f$.
Define
$\delta_h(F)=\sum\limits_{i=2}^{[g/2]} \delta_i(F)$, and
\begin{equation}\label{formulaofdelta_i}
\left\{
\begin{aligned}
&\delta_i(\Upsilon)=\sum_{F\in \Upsilon} \delta_i(F),\quad
\delta_i(\Upsilon_{ct})=\sum_{F\in \Upsilon_{ct}} \delta_i(F),\quad
\delta_i(\Upsilon_{nc})=\sum_{F\in \Upsilon_{nc}} \delta_i(F).\\
&\delta_h(\Upsilon)=\sum_{i=2}^{[g/2]} \delta_i(\Upsilon),\quad
\delta_h(\Upsilon_{ct})=\sum_{i=2}^{[g/2]} \delta_i(\Upsilon_{ct}).
\end{aligned}
\right.
\end{equation}
Then
\begin{equation}\label{formulaofdelta_f}
\left\{
\begin{aligned}
\delta(F)&=\sum_{i=0}^{[g/2]}\delta_i(F)=\delta_0(F)+\delta_1(F)+\delta_h(F),\\
\delta_{f}&=\sum_{i=0}^{[g/2]}\delta_i(\Upsilon)=\delta_0(\Upsilon)+\delta_1(\Upsilon)+\delta_h(\Upsilon).
\end{aligned}
\right.
\end{equation}

\section{Lower bound for $s_{nc}$ when $k=\mathbb C$}\label{sectionthm1}
The section is devoted to proving Theorem \ref{mainthm1}.
We first recall the Miyaoka's inequality and canonical class inequality in Section \ref{sectionmiyaoka}.
The proof of Theorem \ref{mainthm1} is complete in Section \ref{sectionpfofthm1}.
\subsection{Miyaoka's inequality and canonical class inequality}\label{sectionmiyaoka}
We recall Miyaoka's inequality \cite{miyaoka84} as follows.
See also \cite{hirzebruch86} more details.
\begin{theorem}[{\cite[Corollary\,1.3]{miyaoka84}}]\label{theoremmiyaoka}
Let $X$ be a smooth surface over $\mathbb C$ such that the canonical divisor $K_X$ is numerically effective (nef),
$E_1,\cdots,E_n$ be disjoint ADE curves on $X$,
and $D$ be a reduced normal crossing curve on $X$ which does not intersect $E_i$ for each $i$.
Then
\begin{equation}\label{eqnmiyaoka}
\sum_{i=1}^n m(E_i) \leq 3\big(\chit(X)-\chit(D)\big)-(K_X+D)^2.
\end{equation}
where $m(E)$ is defined as follows,
$$\begin{aligned}
m(A_r)&\,=\,3(r+1)-\frac{3}{r+1},\\
m(D_r)&\,=\,3(r+1)-\frac{3}{4(r-2)},\quad\forall~r\geq4,\\
m(E_6)&\,=\,21-\frac{1}{8},\\
m(E_7)&\,=\,24-\frac{1}{16},\\
m(E_8)&\,=\,27-\frac{1}{40}.
\end{aligned}$$
\end{theorem}

Applying Miyaoka's inequality in the case of surfaces fibred by a family of semi-stable curves, one has the following
generalized canonical class inequality.
\begin{lemma}\label{lemmamiyaokafibersurface}
Let $f:\,X \to B$ be a non-isotrivial semi-stable fibration of genus $g\geq 2$
defined over $\mathbb C$ with singular locus $\Upsilon \to \Delta$,
and $D$ be a reduced normal crossing curve on $X$.
Assume that $\Delta=\Delta_1 \cup \Delta_2$ with $\Upsilon_i=f^{-1}(\Delta_i)$, and $D \subseteq \Upsilon_2$.
Then for any $e\in \mathbb N^{*}$ such that $2g(B)-2+\frac{e-1}{e}\cdot |\Delta_1|\geq 0$, we have
\begin{equation}\label{eqnmiyaokafibersurface}
\begin{aligned}
K_f^2 ~\leq&~ (2g-2)\big(2g(B)-2+|\Delta_1|\big)+ 3\delta(\Upsilon_2)-3\chit(D)-(2K_f+D)\cdot D\\
&+\frac{3\delta(\Upsilon_1)}{e^2}-\frac{(2g-2)|\Delta_1|}{e}.
\end{aligned}
\end{equation}
\end{lemma}
\begin{proof}
Let $\pi:\,\wt B \to B$ be a base change of degree $de$ totally ramified over $\Delta_1$ with ramification index $e$.
By Kodaira-Parshin construction (cf. \cite{vojta88}), such a base change exists.
Let $\tilde f:\, \wt X \to \wt B$ be the pull-back fibration,
i.e., it is the relatively minimal smooth model of the fiber-product $X\times_{B}\wt B \to \wt B$.
Let $\Pi: \wt X \to X$ the induced morphism.

$$\xymatrix{
\wt X\ar[r]\ar[dr]_-{\tilde f}\ar@/^5mm/"1,4"^-{\Pi} & X\times_{B}\wt B \ar[rr]\ar[d] && S\ar[d]^-f\\
&\wt B \ar[rr]^-{\pi} && B
}$$

By assumption, $g\big(\wt B\big)\geq 0$. Hence $\wt X$ is minimal of general type.
In particular $K_{\wt X}$ is nef.
Note also that for any node $q$ of $F$ with $F\in \Upsilon_1$,
$\Pi^{-1}(q)$ consists of $d$ disjoint curves of type $A_{e-1}$.
Now applying Theorem \ref{theoremmiyaoka} to $\wt X$ by taking all these inverse image of nodes
of $F$ with $F\in \Upsilon_1$ as ADE curves and $\wt D=\Pi^{-1}(D) \subseteq \wt X$
as the reduced normal crossing curve, one gets
\begin{equation}\label{eqnpfmiyaokafibersurface1}
\delta(\Upsilon_1)\cdot d\left(3e-\frac{3}{e}\right) \leq 3\big(\chit(\wt X)-\chit(\wt D)\big)-
\big(K_{\wt X}+\wt D\big)^2.
\end{equation}
Note that
$$2g(\wt B)-2=de\cdot \left(2g(B)-2+\frac{e-1}{e}\cdot |\Delta_1|\right),$$
$$\delta_{\tilde f}=de\cdot \delta_f,\quad K_{\tilde f}^2=de\cdot K_f^2,\quad \chit(\wt D)=de\cdot \chit(D).$$
Thus \eqref{eqnmiyaokafibersurface} follows from \eqref{eqnpfmiyaokafibersurface1}
together with \eqref{eqnrelativeinvs2}.
\end{proof}

\begin{remark}
Let $\Delta_1=\Delta$, $D=\emptyset$, and $e \rightarrow \infty$ in \eqref{eqnmiyaokafibersurface}.
Then one gets the classic canonical class inequality \cite{vojta88},
$$K_f^2 \leq (2g-2)\big(2g(B)-2+s\big), \qquad s=|\Delta|,$$
which is in fact a strict inequality proved earlier by the second named author \cite{Tan95} for $s\neq 0$ and by Liu \cite{liu96} for $s=0$.
\end{remark}

\begin{corollary}[{\cite{luzuo14}}]\label{corollarycanonical}
Let $f:\,X \to B$ be {as} in Lemma {\rm \ref{lemmamiyaokafibersurface}}.
Then for any $e\in \mathbb N^{*}$ satisfying $2g(B)-2+\frac{e-1}{e}\cdot s_{nc}\geq 0$, one has
\begin{equation}\label{eqncanonical}
K^2_f\leq (2g-2)(2b-2+s_{nc})+2\delta_1(\Upsilon_c)+3\delta_h(\Upsilon_c)+\frac{3\delta(\Upsilon_{nc})}{e^2}-\frac{(2g-2)s_{nc}}{e}.
\end{equation}
\end{corollary}
\begin{proof}
For any $p\in \Delta_c$, let $D_p$ be the union of elliptic tails of $F_p$,
where a curve $D_1\subseteq F_p$ is called an elliptic tail if $F_p=D_1\cup D_2$, $D_1\cdot D_2=1$,
the geometrical genus $g(D_1)=1$, and the component $C$ of $D_2$ intersecting $D_1$ is not a rational curve
with $C^2=-2$.
It is easy to see that the number of irreducible component of $D_p$ is $l(D_p)=\delta_1(F_p)$.

Applying Lemma \ref{lemmamiyaokafibersurface} to the family $f:\,X \to B$ by setting $\Delta_1:=\Delta_{nc}$
and $D:=\bigcup\limits_{p\in \Delta_c} D_p,$
one obtains
\begin{equation*}
\begin{aligned}
K_f^2 ~\leq&~ (2g-2)(2b-2+s_{nc})+\frac{3\delta(\Upsilon_{nc})}{e^2}-\frac{(2g-2)s_{nc}}{e}\\
&+3\delta(\Upsilon_c)-3\chit(D)-(2K_f+D)\cdot D.
\end{aligned}
\end{equation*}
Since $\delta(\Upsilon_c)=\delta_1(\Upsilon_c)+\delta_h(\Upsilon_c)$ and $l(D_p)=\delta_1(F_p)$, it suffices to prove
\begin{equation}\label{eqnpfofcanonical1}
3\chit(D_p)+(2K_f+D_p)\cdot D_p\geq l(D_p),\qquad \forall~p\in \Delta_c.
\end{equation}
Since $F_p$ has a compact Jacobian, each connected component of $D_p$ consists of a chain of rational curves
plus one elliptic curves.
Let $C_{p}$ be such a connected component. Then
$\chit(C_p)=l(C_p)-1$, $K_f\cdot C_p=1$ and $C_p^2=-1$.
Hence
$$3\chit(C_p)+(2K_f+C_p)\cdot C_p=3l(C_p)-2 \geq l(C_p).$$
Therefore, \eqref{eqnpfofcanonical1} is proved; and hence the proof is complete.
\end{proof}

\subsection{Proof of Theorem \ref{mainthm1}}\label{sectionpfofthm1}
We prove Theorem \ref{mainthm1} by contradiction in the section.
Let $f:\,X \to \bbp^1$ be a non-isotrivial semi-stable fibration of genus $g\geq 2$
defined over $\mathbb C$.
According to \cite{V-Z04} or \cite{LTYZ}, it is known that $s_{nc} \geq 4$.
In order to derive a contradiction,
we assume that $g\geq 5$ and $s_{nc}=4$ in this section.

Let $q:=h^1(X,\mathcal {O}_X)$ be the irregularity of $X$.
By \cite[Corollary\,1.12]{LTYZ}, we have $q\leq1$ and
\begin{equation}\label{eqnpfofmain11}
\chi_f=g-q
\end{equation}
Applying Corollary \ref{corollarycanonical} to our family $f:\,X \to \mathbb P^1$,
one gets
\begin{equation}\label{eqnpfofmain12}
K^2_f\leq 4(g-1)+2\delta_1(\Upsilon_c)+3\delta_h(\Upsilon_c)+\frac{3\delta(\Upsilon_{nc})}{e^2}-\frac{8(g-1)}{e},
\qquad \forall~2\leq e\in \mathbb N^{*}.
\end{equation}
According to \cite[Theorem 4.2]{luzuo14}, which is essentially due to Moriwaki \cite[Theorem D]{mo98}, we obtain
\begin{equation}\label{eqnpfofmain13}
K^2_f\geq \frac{4(g-1)}{g}\cdot \chi_f + \frac{3g-4}{g}\delta_1(\Upsilon)+\frac{7g-16}{g}\delta_h(\Upsilon).
\end{equation}

Consider first the case $q=0$.
Note that $\delta_1(\Upsilon_c) \leq \delta_1(\Upsilon)$ and $\delta_h(\Upsilon_c)\leq \delta_h(\Upsilon)$ by the definition.
So it follows from \eqref{eqnpfofmain11} together with \eqref{eqnpfofmain12} and \eqref{eqnpfofmain13} that
\begin{equation}\label{eqnpfofmain14}
\frac{g-4}{g}\left(\delta_1(\Upsilon)+4\delta_h(\Upsilon)\right) \leq \frac{3\delta(\Upsilon_{nc})}{e^2}-\frac{8(g-1)}{e} <0,
 \quad \text{if $e$ is large enough}.
\end{equation}
Since both $\delta_1(\Upsilon)$ and $\delta_h(\Upsilon)$ are non-negative and $g\geq 5$,
we get a contradiction from \eqref{eqnpfofmain14} (Actually, we can still get a contradiction even if $g=4$ in the case).
Thus Theorem \ref{mainthm1} is proved for the case $q=0$.\vspace{0.1cm}

Now we may assume that $q=1$.
Similarly, by \eqref{eqnpfofmain11}, \eqref{eqnpfofmain12} and \eqref{eqnpfofmain13}, we obtain
\begin{equation}\label{eqnpfofmain15}
0\geq \frac{g-4}{g}\left(\delta_1(\Upsilon)+4\delta_h(\Upsilon)\right)-\frac{4(g-1)}{g}
+\frac{8(g-1)}{e} -\frac{3\delta(\Upsilon_{nc})}{e^2}.
\end{equation}
By \cite{Yu10}, the number of singular fibre $s\geq 6$ since $g\geq 5>2$.
Hence $\delta_1(\Upsilon)+\delta_h(\Upsilon)\geq 2$.
Thus $K_f^2 > 4(g-1)$ by \eqref{eqnpfofmain13}.
Combining this with Lemma \ref{lemmapfofmain11} blow, one gets $K_f^2\geq 5g-6+q=5(g-1)$.
Hence $\delta(\Upsilon_{nc}) \leq \delta_f=12\chi_f-K_f^2 \leq 7(g-1)$.
Letting $e=g$ in \eqref{eqnpfofmain15}, one gets
\begin{eqnarray*}
0&\geq&\frac{g-4}{g}\left(\delta_1(\Upsilon)+4\delta_h(\Upsilon)\right)+\frac{4(g-1)}{g}-\frac{3\delta(\Upsilon_{nc})}{g^2}\\
&\geq&\frac{2(g-4)}{g}+\frac{4(g-1)}{g}-\frac{21(g-1)}{g^2}=\frac{3(2g^2-11g+7)}{g^2},
\end{eqnarray*}
which is a contradiction since $g\geq 5$.
The proof is complete.
\qed

\begin{lemma}\label{lemmapfofmain11}
Let $f:\,X \to \bbp^1$ be a relatively minimal fibration defined over $\mathbb C$. Then
$K_f^2\geq 4(g-1)$; moreover if $K_f^2\neq 4(g-1)$, then $K_f^2\geq 5g-6+q$, where $q=h^1(X,\mathcal {O}_X)$ is
the irregularity of X.
\end{lemma}
\begin{proof} (Zamora, Alexis G.)
The first statement is proved in \cite[Theorem 2.1]{TTZ05} based on the fact that $A:=K_X+F$ is nef,
where $F$ is a general fibre of $f$.

Note that $K_f^2=A^2+4(g-1)$.
Hence if $K_f^2\neq 4(g-1)$, then $A$ is nef and big, which implies $H^i(K_X+A)=0$ for $i=1,2$ by Mumford vanishing theorem.
Hence by Riemann-Roch theorem, one has
$$0\leq h^0(K_X+A)=\chi(K_X+A)=\chi(\mathcal O_X)+\frac12(K_X+A)\cdot A=K_f^2-(5g-6+q).
$$
This completes the proof.
\end{proof}

\section{Lower bound for $s_{nc}$ in the positive characteristic case}\label{sectionthm2}
The main purpose of the section is to prove Theorem \ref{thm2}.
Hence, we always assume that the characteristic of $k$ is positive.
First, we have the following easy lemma:
\begin{lemma}\label{lemma1}
Let $f:X \rightarrow \mathbb{P}^1 $ be a  non-trivial semi-stable fibration of genus $g\geq 1$ over an algebraically closed field $k$. Then
\begin{equation}\label{eq0}
b_2-\left(2+\sum\limits_{F\in\Upsilon}\big(l(F)-1\big)\right)=\left(g-\frac{b_1}{2}\right)\cdot(s_{nc}-4)- \sum\limits_{F\in\Upsilon_{nc}}\left(g(F)-\frac{b_1}{2}\right),
\end{equation}
where $l(F)$ is the number of irreducible components in $F$, and $b_1 , b_2$ are the first and second Betti numbers of $X$.
\end{lemma}

\begin{proof}
Combining \eqref{eqnrelativeinvs2} with \eqref{eqndelta(F)=}, one gets
\begin{displaymath}
2-2b_1+b_2+4g-4=\delta_f=\sum\limits_{F\in\Upsilon_{nc}}\big(g-g(F)\big)+\sum\limits_{F\in\Upsilon}\big(l(F)-1\big).
\end{displaymath}
By rearrangement, we obtain \eqref{eq0}.
\end{proof}

Before going further, let's recall some facts about the Albanese variety.
It is well-known that the Picard scheme $\Pic^0(P)$ and the Albanese variety $\text{Alb}(P)$ exist for any nonsingular projective variety $P$.
For the construction, we refer to \cite[\S\,5]{Ba01}.
We just remark here that
\begin{enumerate}
\item $\text{Alb}(C)$ is isomorphic to the Jacobian $J(C)$ if $P=C$ is a curve;
\item the Zariski tangent space of $\Pic^0(P)$ at the origin is canonically isomorphic to $H^1(P,\mathcal O_P)$;
\item $\text{Alb}(P)$ is the dual abelian variety of $\Pic^0(P)_{red}$,
where $\Pic^0(P)_{red}$, called the (classical) Picard variety of $P$, is the the reduced part of $\Pic^0(P)$.
\end{enumerate}
In particular, one has
\begin{equation}
q:=\dim H^1(P,\mathcal O_P) \geq\frac{b_1}{2}=\text{dim\,} \text{Alb}(P)=\text{dim\,} {\rm Pic}^0(P).
\end{equation}
It is a little different from the characteristic zero case where we always have $q=\frac{b_1}{2}$.
In fact, there do exist surfaces in the positive characteristic case with $q>\frac{b_1}{2}$, see \cite{Ig55} and \cite{Se58a}.

The {following} lemma, whose proof is an analogue of \cite[Lemma\,4.1]{LTYZ},
is a generalization of \cite[Lemma\,1]{Be81} in characteristic zero case.
\begin{lemma}\label{lemma2}
Let $f:\,X \to \mathbb P^1$ be as in Lemma {\rm\ref{lemma1}}, $F$ any fiber of $f$, and $g(F)$. the geometrical genus of $F$.
Then $g(F)\geq \frac{b_1}{2}$.
\end{lemma}
\begin{proof}
Let $\widetilde{F}$ be the normalization of $F$ and $\wt F \to X$ the induced morphism.
By the universal property of the Albanese variety, we have a natural map $\beta: J(\widetilde{F})\to {\rm Alb}(X)$.
Denote by $A={\rm \text{Alb}}(X)/{\rm Im}\beta$ the quotient abelian variety, and $\overline{\alpha}:X\rightarrow A$ the induced map.
$\overline{\alpha}(F)$ is a point in $A$ since $J(\widetilde{F})$ maps to a point in $A$.
Therefore,  by the rigidity theorem, $\overline{\alpha}$ contracts all fibers of $f$.
So $\overline{\alpha}$ factors through $f$.
\begin{equation*}
\raisebox{15mm}{\xymatrix{
X\ar[r]^{\overline{\alpha}}\ar[d]_{f} & A   \\
\mathbb{P}^1\ar[ur]_{\varphi} &
}}
\end{equation*}
Since the image of $X$ in \text{Alb}$(X)$ generates \text{Alb}$(X)$,
the image of $X$ in $A$ generates $A$, i.e., $A$ is generated by $\bar\alpha(X)=\varphi(\mathbb P^1)$.
However, it is well-known that an abelian variety cannot contain any rational curves.
Hence $\varphi(\mathbb P^1)$ is a point, which implies $\dim A=0$.
Equivalently, $\beta: J(\widetilde{F})\to {\rm Alb}(X)$ is surjective.
Therefore, $g(F)=\dim J(\widetilde{F}) \geq \frac{b_1}{2} =\dim {\rm Alb}(X)$ as required.
\end{proof}

Now we are able to prove Theorem \ref{thm2}.
\begin{proof}[{\it Proof of Theorem {\rm \ref{thm2}}}]
Let $\rho$ be the rank of the N\'{e}ron-Severi group $NS(X)$.
Then we have $b_2\geq \rho$ by Igusa \cite{Ig60} and $\rho \geq \left(2+\sum\limits_{F\in\Upsilon}\big(l(F)-1\big)\right)$ by Beauville \cite{Be81}.
Combining these with Lemmas \ref{lemma1} and \ref{lemma2}, we get
\begin{equation}\label{eqnpfofthm21}
\left(g-\frac{b_1}{2}\right)\cdot(s_{nc}-4)\geq \sum\limits_{F\in\Upsilon_{nc}}\left(g(F)-\frac{b_1}{2}\right)\geq 0.
\end{equation}
Since the relative Jacobian of $f$ is non-smooth, $\Upsilon_{nc}\neq \emptyset$.
By Lemma \ref{lemma2}, one has $g>g(F)\geq \frac{b_1}{2}$ for any fiber $F\in\Upsilon_{nc}$.
Therefore, it follows from \eqref{eqnpfofthm21} that $s_{nc}\geq 4$.
\end{proof}

Recall that an algebraic surface $X$ is said to be Shioda-supersingular if $b_2=\rho$.
The above proof also shows the following proposition.
\begin{proposition}
Let $f:\,X \to \mathbb P^1$ be as in Lemma {\rm\ref{lemma1}}. Then either $s_{nc}\geq 4$ or $g=\frac{b_1}{2}$. Moreover,\vspace{0.1cm}

{\rm (i)}
if $s_{nc}=4$, then $X$ is Shioda-supersingular, $g(F)=\frac{b_1}{2}$ for $F\in\Upsilon_{nc}$, and
\begin{equation}\label{eqnprop1}
b_2=2+\sum_{F\in\Upsilon}(l(F)-1);
\end{equation}

{\rm (ii)} if $g=\frac{b_1}{2}$, then $X$ is Shioda-supersingular, the relative Jacobian of $f$ is a non-trivial smooth family of abelian varieties over $\mathbb P^1$,
and \eqref{eqnprop1} holds.
The image of the Albanese map of $X$ is a surface if moreover $g\geq 2$.
\end{proposition}

\begin{proof}
It follows from \eqref{eqnpfofthm21} that $s_{nc}\geq 4$ or $g=\frac{b_1}{2}$.
And if either $s_{nc}=4$ or $g=\frac{b_1}{2}$, one has
$$b_2=\rho=2+\sum_{F\in\Upsilon}(l(F)-1).$$
Hence $X$ is Shioda-supersingular and \eqref{eqnprop1} holds.
It follows also from the proof of Theorem \ref{thm2} that the relative Jacobian of $f$ is smooth if $g=\frac{b_1}{2}$,
and $g(F)=\frac{b_1}{2}$ for $F\in\Upsilon_{nc}$ if $s_{nc}=4$,
due to the fact that $g\geq g(F)\geq \frac{b_1}{2}$ for any fiber $F$ by Lemma \ref{lemma2}.
It remains to prove that if $g=\frac{b_1}{2}$, then the image of the Albanese map is a surface.

Assmue that $g=\frac{b_1}{2}$, and let $\alpha:\, X \rightarrow \text{Alb}(X)$ be the Albanese map and $F$ be a general fiber of $f$.
Suppose that ${\rm Im}(\alpha)=C$ is a curve. Then $C$ is smooth of genus $g(C)=\dim \text{Alb}(X)=g$. Consider the following diagram.
$$\xymatrix@M=0.15cm{
F \ar@{^(->}[rr] \ar[drr]_{\pi} &&X \ar[rr]^{(\alpha,f)} \ar[d]^{\alpha}&& C \times \mathbb P^1\\
&&C &&
}$$
It is clear that $\pi:\,F \to C$ is surjective.
Note that $\pi$ is a composition of a purely inseparated morphism $\pi_1:\,F \to F'$ and a separated morphism $\pi_2:\,F' \to C$;
moreover, $F'$ is isomorphic to $F$ as abstract schemes and $\pi_1$ is a composition of Frobenius morphisms (cf. \cite[\S\,IV-2]{Ha77}).
Since $f$ is non-isotrivial, one obtains that $\deg \pi_2 =d \geq 2$.
Thus
$$2g-2 =2g(F')-2\geq d\cdot (2g(C)-2)\geq 2(2g-2),$$
which is a contradiction, since $g\geq 2$. The proof is complete.
\end{proof}

\begin{corollary}
Let $f:\,X \to \mathbb P^1$ be as in Lemma {\rm\ref{lemma1}}. If $X$ is not Shioda-supersingular, then $s\geq s_{nc}\geq5$.
\end{corollary}

\vspace{0.2cm}
\noindent{\it Acknowledgements}:
We would like to thank Abolfazl Mohajer for his interest and a careful reading.


\begin{thebibliography}{ABCD}

\bibitem[Ba01]{Ba01} B$\rm \breve{a}$descu, L.
{\it Algebraic surfaces}.  Universitext. Springer-Verlag, New York, 2001.

\bibitem[BHPV]{bhpv}  Barth, W. P.; Hulek, K.; Peters, C. A. M.; Van de Ven, A.
{\it Compact complex surfaces}. Second edition.
A Series of Modern Surveys in Mathematics,
Volume 4, Springer-Verlag, 2004.

\bibitem[Be81]{Be81} Beauville, A.
{\it Le nombre minimum de fibres singuli\`{e}res d'un courbe stable sur $\mathbb{P}^1$}.
In: S\'eminaire sur les Pinceaux de Courbes de Genre au Moins Deux. (French).
Ast\'{e}risque 86, Soci\'et\'e Math\'ematique de France, Paris, 1981.  97--108.

\bibitem[Be82]{Be2} ------,
{\it Les familles stables de courbes elliptiques sur $\mathbb P^1$ admettant quatre fibres singuli\`eres}. (French).
C. R. Acad. Sci. Paris S\'er. I Math. 294 (1982), no. 19, 657--660.

\bibitem[GLT]{glt13} Gong, C.; Lu, X.; Tan, S.-L.
{\it Families of curves over $\mathbb P^1$ with $3$ singular fibers}.
C. R. Math. Acad. Sci. Paris 351 (2013), no. 9-10, 375--380.

\bibitem[Ha77]{Ha77} Hartshorne, R.
{\it Algebraic geometry}.
Graduate Texts in Mathematics, No. 52. Springer-Verlag, New York-Heidelberg, 1977.

\bibitem[Hi86]{hirzebruch86} Hirzebruch, F.
{\it Singularities of algebraic surfaces and characteristic numbers}.
The Lefschetz centennial conference, Part I (Mexico City, 1984), 141--155,
Contemp. Math., 58, Amer. Math. Soc., Providence, RI, 1986.

\bibitem[Ig55]{Ig55} Igusa, J.-I.
{\it On some problems in abstract algebraic geometry}.
Proc. Nat. Acad. Sci. U. S. A. 41 (1955), 964--967.

\bibitem[Ig60]{Ig60} ------,
{\it Betti and Picard numbers of abstract algebraic surfaces}.
Proc. Nat. Acad. Sci. U.S.A. 46 (1960) 724--726.

\bibitem[Ku10]{Ku10} Kukulies, S.
{\it On Shimura curves in the Schottky locus}.
J. Algebraic Geom. 19 (2010), no. 2, 371--397.

\bibitem[Li96]{liu96} Liu, K.-F.
{\it Geometric height inequalities}.
Math. Res. Lett.  3  (1996),  no. 5, 693--702.

\bibitem[LTYZ]{LTYZ} Lu, J.; Tan, S.-L.; Yu, F.; Zuo, K.
{\it A new inequality on the Hodge number $h^{1,1}$ of algebraic surfaces}.
Math. Z. 276 (2014), no. 1-2, 543--555.

\bibitem[LZ14]{luzuo14} Lu, X.; Zuo, K.
{\it The Oort conjecture on Shimura curves in the Torelli locus of curves}, (2014), arXiv:1405.4751.

\bibitem[Mi84]{miyaoka84} Miyaoka, Y.
{\it The maximal number of quotient singularities on surfaces with given numerical invariants}.
Math. Ann. 268 (1984), no. 2, 159--171.

\bibitem[MB81]{M-B} Moret-Bailly, L.
{\it Familles de courbes et de vari\'{e}t\'{e}s ab\'{e}liennes sur $\mathbb{P}^1$ II. exemples}.
In: S\'eminaire sur les Pinceaux de Courbes de Genre au Moins Deux. (French).
Ast\'{e}risque 86, Soci\'et\'e Math\'ematique de France, Paris, 1981.  125--140.

\bibitem[Mo98]{mo98} Moriwaki, A.
{\it Relative Bogomolov's inequality and the cone of positive divisors on the moduli space of stable curves}.
J. Amer. Math. Soc.  11  (1998),  no. 3, 569--600.

\bibitem[Ng98]{Ngu98} Nguyen K. V.
{\it A remark on semi-stable fibrations over $\mathbb{P}^1 $ in positive characteristic}.
Compositio Math. 112 (1998), no. 1, 41--44.

\bibitem[Se58]{Se58a} Serre, J.-P.
{\it Sur la topologie des vari\'et\'es alg\'ebriques en caract\'eristique p}. (French).
1958 Symposium internacional de topolog\'ia algebraica International symposium on algebraic topology pp. 24--53
Universidad Nacional Aut¨®noma de M¨¦xico and UNESCO, Mexico City.

\bibitem[Sz81]{Sz81} Szpiro, L.
{\it Sur les propri\'et\'es num\'eriques du dualisant relatif d'une surface arithm\'etique}. (French)
The Grothendieck Festschrift, Vol. III, 229--246, Progr. Math., 88, Birkh\"auser Boston, Boston, MA, 1990.

\bibitem[Ta95]{Tan95} Tan, S.-L.
{\it The minimal number of singular fibers of a semistable curve over $\mathbb{P}^1$}.
J. Algebraic Geom. 4 (1995), no. 3, 591--596.

\bibitem[TTZ]{TTZ05} Tan, S.-L.; Tu, Y.; Zamora, A. G.
{\it On complex surfaces with 5 or 6 semistable singular fibers over $\mathbb{P}^1$}.
Math. Z. 249 (2005), no. 2, 427--438.

\bibitem[VZ04]{V-Z04} Viehweg, E.; Zuo, K.
{\it A characterization of certain Shimura curves in the moduli stack of abelian varieties},
J. Differential Geom. 66 (2004), no. 2, 233--287.

\bibitem[Vo88]{vojta88} Vojta, P.
{\it Diophantine inequalities and Arakelov theory}.
In: Lang, Introduction to Arakelov theory. Springer-Verlag, New York, 1988, 155--178.

\bibitem[Yu10]{Yu10} Yu, F.
{\it Note on families of semistable curves over $\mathbb{P}^1$ with $4$ singular fibers whose Jacobian are non-compact}.
Sci. China Math. 53 (2010), no. 7, 1711--1714.
\end{thebibliography}
\end{document}